\documentclass{amsart}
\usepackage{graphicx}
\usepackage{verbatim}

\usepackage{amsmath, amssymb, amsthm}
\usepackage{amsfonts}
\newtheorem{theorem}{Theorem}
\newtheorem{lemma}[theorem]{Lemma}

\numberwithin{equation}{section}
\newtheorem{proposition}[theorem]{Proposition}

\newtheorem*{acknowledgements}{Acknowledgements}

\begin{document}

\title{A unified proof of Brooks' theorem and Catlin's theorem}
\author{Vaidy Sivaraman}

\address{Department of mathematical sciences, Binghamton University.}
\email{vaidy@math.binghamton.edu}

\begin{abstract}
 We give a unified proof of Brooks' theorem and Catlin's theorem. \end{abstract}

\date{February 26, 2014 \\  \text{      }  \text{  } \small {2010 Mathematics Subject Classification: 05C15}}   

\maketitle

All graphs in this note are simple and finite. Let $G$ be a graph.  An $n$-coloring of $G$ is a partition of $V(G)$ into $n$ independent sets. The chromatic number of $G$ is denoted by $\chi(G)$. The independence (stability) number of $G$ is denoted by $\alpha(G)$. If $X$ is a set of vertices of $G$, then $G \backslash X$ denotes the graph obtained from $G$ by deleting the vertices in $X$. We follow the notation and terminology of West \cite{DBW}.\\

Several proofs of Brooks' theorem appear in the literature, the most famous one being \cite{LL} (also see \cite{DBW}). 
There is a strengthening of Brooks' theorem, due to Catlin \cite{PAC}, which states that every graph $G$ with maximum degree $d \geq 3$ and no clique of size $d + 1$, has a $d$-coloring in which one of the color classes has size $\alpha(G)$. \\ 

We begin by proving Catlin's result for triangle-free subcubic graphs, and the general result follows from that by induction. In addition to Catlin's original proof, several other proofs of Catlin's theorem are known (\cite{HT}, \cite{JM}, \cite{CL}). Also, our proof has some similarity to the recent one by L. Rabern (\cite{LR}), but the focus there is only on Brooks' theorem.


\begin{theorem} Let $d$ be an integer at least 3, and let $G$ be a graph with
maximum degree $d$. If $G$ does not contain $K_{d+1}$ as a subgraph, then $G$
has a $d$-coloring in which one color class has size $\alpha(G)$. In
particular, $\chi(G)\le d$. 
\end{theorem}
\begin{proof} Our proof is by induction on $|V(G)|$. We consider induction
steps when either $d\ge 4$ or $G$ contains a copy of $K_d$. Thus, our base case
is when $d=3$ and $G$ contains no triangles. 

{\bf Base case.} Choose an independent set $I$ of size $\alpha(G)$ such that
the number of odd cycles in $G\setminus I$ is minimum. Suppose that
$G\setminus I$ contains an odd cycle $C$. Choose $v\in V(C)$. Consider the
set $S$ of all paths $P$ starting at $v$ and alternating between non-isolated
vertices of $G\setminus I$ and vertices of $I$, subject to $V(G\setminus I)\cap
V(P)$ being independent. Let $P_0$ be a member of $S$ of maximum length, and
let $I'$ be the symmetric difference of $I$ and $V(P_0)$. Note that $I'$ is
independent. This is because $d=3$, $v$ has a neighbor in $I\cap V(P)$ and two
neighbors in $G\setminus I$, and every other vertex of $I'\setminus I$ has two
neighbors in $I\cap V(P)$ and one neighbor in $G\setminus I$. Since $I$ is
maximum and $P$ starts outside $I$, we have $|I'|=|I|=\alpha(G)$.

Note that no cycle in $G\setminus I'$ contains a vertex of $I\cap V(P_0)$.
This is because (by construction) each vertex of $I$ in the interior of $P_0$
has two neighbors in $I'$, and the final vertex of $P_0$ has each neighbor
either in $I'$ or isolated in $G\setminus I$. This means that every odd cycle
in $G\setminus I'$ is an odd cycle in $G\setminus I$. Since $C$ is an odd
cycle of $G\setminus I$ that is not in $G\setminus I'$, the number of odd
cycles in $G\setminus I'$ is strictly less than that in $G\setminus I$, a
contradiction. Hence $G\setminus I$ contains no odd cycle; therefore
$G\setminus I$ is bipartite and can be 2-colored. Coloring $I$ with color 3
gives the desired 3-coloring with color class 3 of size $\alpha(G)$.

{\bf Induction step.} Suppose first that $d\ge 4$ and $G$ contains no copy of
$K_d$. Let $I$ be a maximum independent set. Now $G\setminus I$ has maximum
degree at most $d-1$, so by induction, $G\setminus I$ has a $(d-1)$-coloring.
Using color $d$ on $I$ gives the desired coloring.

Suppose instead that $G$ contains a copy of $K_d$. Let $\{v_1, \ldots, v_d\}$ be a set of pairwise
adjacent vertices. For each $v_i$, let $a_i$ be its neighbor outside
$\{v_1,\ldots,v_d\}$ and let $A=\{a_i : 1\le i\le d\}$ (the case where some
$v_i$ has no such neighbor $a_i$ is easier, and we consider it below). The
$a_i$ need not be distinct, but cannot all be equal, since $G$ does not contain
$K_{d+1}$. Let $G'=G\setminus\{v_1,\ldots,v_d\}$. Our plan is to color $G'$
by hypothesis, then extend the coloring to $G$. To do so, we must ensure that
the $a_i$ do not all receive the same color.

First suppose that $G'$ has a maximum independent set $I$ not containing all of
$A$. Form $G''$ from $G'$ by adding some edge $a_ia_j$ where $a_i$ is not in
$I$. By induction, we have a $d$-coloring of $G''$ with some color class of
size $\alpha(G'')$. We can easily extend this coloring to the desired
$d$-coloring of $G$. For example, we can apply Hall's Theorem to a bipartite graph with one
part consisting of $\{v_1,\ldots,v_d\}$, the other part consisting of colors
$\{1,\ldots,d\}$, and each vertex $v_i$ adjacent to all colors not used on
$a_i$. The largest color class in this $d$-coloring of $G$ has size
$1+\alpha(G'')=1+\alpha(G')=\alpha(G)$; the final equality holds because every
independent set in $G$ contains at most one vertex of $\{v_1,\ldots,v_d\}$.

If instead every maximum independent set of $G'$ contains all of $A$, then form
$G''$ from $G'$ by adding an arbitrary edge $a_ia_j$. Again, we apply the
induction hypothesis to $G''$ and extend the $d$-coloring of $G''$ to $G$ by
Hall's Theorem. Now the largest color class has size
$1+\alpha(G'')=\alpha(G')=\alpha(G)$.

The case where some $v_i$ has no neighbor $a_i$ is easier. After $d$-coloring
$G'$ by induction, extending the coloring is simple, since $v_i$ can receive
any color. \end{proof}

\begin{acknowledgements}
I am grateful to Prof. Douglas West who carefully read this short note and suggested several improvements in the presentation. 
Also, I would like to thank the anonymous referee who completely rewrote the proof and it is his/her write-up that is presented here.\end{acknowledgements}

\end{document}